\documentclass[11pt]{article}

\usepackage{latexsym,mathrsfs}
\usepackage{amsmath,amssymb,mathtools} 
\usepackage{amsthm,enumerate,verbatim}
\usepackage{amsfonts}
\usepackage{graphicx}
\usepackage{algorithm}
\usepackage{algpseudocode}
\usepackage{bm}
\usepackage{url}
\usepackage{placeins}
\usepackage[makeroom]{cancel}
\usepackage[thinlines]{easytable}
\usepackage{bbold}
\usepackage{nicefrac} 

\usepackage[savepos]{zref}
\usepackage{multirow}

\usepackage{chngcntr}
\counterwithout{table}{section}

\usepackage{tikz}
\usetikzlibrary{fit,
	arrows,
	backgrounds,
	calc,
	positioning,
	shadings,
	shapes.arrows,
	shapes.symbols,
	shapes.geometric,
	shadows,
	decorations.pathmorphing,
	decorations.pathreplacing,
	matrix,
	chains,
	automata,
	babel}

\newcommand{\comp}[1]{$\mathcal{O}(#1)$}

\newtheorem{lemma}{Lemma}

\newtheorem{theorem}{Theorem}

\newtheorem*{definition*}{Definition}
\newtheorem{example}{Example}

\DeclareMathOperator{\argmax}{argmax} 
\DeclareMathOperator{\argmin}{argmin}

\def\Li{\par\leavevmode\par}

\title{Off-diagonal Symmetric Nonnegative Matrix Factorization} 
\date{}

\author{Fran\c{c}ois Moutier, Arnaud Vandaele, Nicolas Gillis\thanks{This work is supported by the Fonds de la Recherche Scientifique - FNRS and the Fonds Wetenschappelijk Onderzoek - Vlanderen (FWO) under EOS Project no O005318F-RG47, and by the  European  Research  Council (ERC  starting  grant  no  679515). 
Emails: \{francois.moutier,arnaud.vandaele,nicolas.gillis\}@umons.ac.be.} \\ 
	Department of Mathematics and Operational Research \\ 
	Facult\'e Polytechnique, Universit\'e de Mons \\ 
	Rue de Houdain 9, 7000 Mons, Belgium   
}

\begin{document}
	
	\maketitle

	\begin{abstract}
Symmetric nonnegative matrix factorization (symNMF) is a variant of nonnegative matrix factorization (NMF) that allows to handle symmetric input matrices and has been shown to be particularly well suited for clustering tasks.	
	In this paper, we present a new model, dubbed off-diagonal symNMF (ODsymNMF), that does not take into account the diagonal entries of the input matrix in the objective function. 
	ODsymNMF has three key advantages compared to symNMF. 
	First, ODsymNMF is theoretically much more sound as there always exists an exact factorization of size at most $\nicefrac{n(n-1)}{2}$ where $n$ is the dimension of the input matrix. Second, it makes more sense in practice as diagonal entries of the input matrix typically correspond to the similarity between an item and itself, not bringing much  information. 
	Third, it makes the optimization problem much easier to solve. In particular, it will allow us to design an algorithm based on coordinate descent that minimizes the component-wise $\ell_1$ norm between the input matrix and its approximation. 
	We prove that this norm is much better suited for binary input matrices often encountered in practice. We also derive a coordinate descent method for the component-wise $\ell_2$ norm, and compare the two approaches with symNMF on synthetic and document data sets. 
	\end{abstract}
	

	\section{Introduction}
	
	Nonnegative matrix factorization (NMF) is a widely used  linear dimension reduction (LDR) technique which extracts useful information in images, documents, or more generally nonnegative data sets. 
	Given a nonnegative matrix $A \in \mathbb{R}^{m\times n}_+$ and a positive integer $r < \min(m,n)$, NMF aims at finding two nonnegative matrices $W\in \mathbb{R}^{m\times r}_+$ and $H\in \mathbb{R}^{n\times r}_+$ such that the low-rank matrix $WH^T$ approximates the input matrix $A$, which means that $A_{ij}\approx (WH^T)_{ij}$ for $i=1,...,m$ and $j=1,...,n$.
	The design and algorithmic implementation of refined NMF models for various applications is still a very active area of research; see~\cite{cichocki2009nonnegative,  xiao2019uniq,gillis2014and} and the references therein. 
	
	When the input matrix $A \in \mathbb{R}^{n\times n}_+$ is symmetric, it makes sense to look for a low-rank approximation which is symmetric as well.
	For this purpose,  symmetric nonnegative matrix factorization (symNMF) seeks a matrix $H\in \mathbb{R}^{n\times r}_+$ such that $HH^T$ approximates $A$, that is
	$A_{ij}\approx (HH^T)_{ij}$ for $1 \leq i,j \leq n$. 
	SymNMF is mainly used as a clustering method.
	In fact, the matrix $A$ usually represents the similarity measured between each pair of a set of $n$ elements. 
	The symNMF $HH^T$ of $A$ amounts to decomposing $A$ into $r$ rank-one factors 
	\[
	A \; \approx \; HH^T =  
	\sum_{k=1}^r H_{:,k}{H_{:,k}}^T. 
	\]
Since the rank-one factors are nonnegative, there is no cancellation and $A$ is approximated via the sum of $r$ rank-one nonnegative matrices. 
The non-zero entries of a rank-one factor correspond to a square submatrix of $A$ with mostly positive entries, that is, to a cluster within $A$ where all elements are highly connected. 	
	SymNMF has been used successfully in many different settings and was proved to compete with standard clustering techniques such as normalized cut, spectral clustering, k-means and spherical k-means; 
	see \cite{chen2008non,kuang2012symmetric, kuang2015symnmf,  long2007relational, yan2013learning, yang2012clustering, zass2005unifying} and the references therein.
	
	In order to find the matrix $H$, the symNMF problem is mainly tackled by solving the following optimization problem
	\begin{equation}
	\min_{H\geq 0}  \quad ||A-HH^T||^2_F,\label{symmodel}
	\end{equation}
	which is non-convex and NP-hard to solve~\cite{dickinson2014computational}. 
	Nevertheless, several local schemes were developed in order to obtain acceptable solutions--typically such algorithms are guaranteed to converge to first-order stationary points of~\eqref{symmodel}; see for example \cite{huang2013non, kuang2015symnmf, shi2017inexact,vandaele2016efficient}.

\paragraph{Outline and contribution of the paper} 
	
	In this work, we introduce a closely related variant of symNMF where the diagonal entries of the input matrix $A$ are not taken into account, that is, we are looking for a low-rank approximation $HH^T$ such that 
	\begin{equation}\label{odsymnmf}
	A_{ij} \approx (HH^T)_{ij} \quad \text{ for } i\neq j.
	\end{equation} 
It has to be noted that this idea has already been used in the context of approximation of correlation matrices~\cite{borsdorf2010computing}. However, the nonnegativity of the factor $H$ is not enforced, hence the problem is rather different, being a symmetric eigenvalue problem efficiently solvable. 
	
	Throughout this paper, we will refer to this problem as off-diagonal SymNMF (ODsymNMF) and focus on solving 
	\begin{equation}\label{odgeneral}
	\min_{H \geq 0} \|A-HH^T\|_{\text{OD},p} 
	\; \text{ where } \;  
			\|A-HH^T\|_{\text{OD},p} = \left(\sum_{i=1}^n\sum_{\substack{j=1\\j\neq i}}^n \left(A-HH^T\right)_{ij}^p\right)^{\frac{1}{p}}.
	\end{equation} 	
	Although this model might be surprising at first (one may say odd), we describe its advantages and why it is meaningful in practice in Section~\ref{Models}.
	In Section \ref{algorithms}, we develop two local  algorithms based on coordinate descent (CD) to tackle the cases $p=1$ and $p=2$. 
In Section~\ref{Initialization}, we propose an    initialization  scheme for ODsymNMF that is particularly crucial when $p=1$ as it is more sensitive to initialization than when $p=2$. 
	In Section~\ref{Numeric}, we perform some numerical experiments on synthetic and real examples (document data sets) highlighting the validity of the ODsymNMF model.

	\section{The why of ODsymNMF} \label{Models}

	In this section, we discuss the advantages of ODsymNMF compared to symNMF.   
	We also show that ODsymNMF for $p=1$ is an ideal model  in the rank-one case when $A$ is binary.

	\subsection{Advantages of ODsymNMF}

Let us describe the three most important advantages of 	ODsymNMF compared to symNMF.  
	
	\paragraph{From a practical point a view.} When the entries of $A$ correspond to the similarity between items, the detection of clusters is made more complicated by the overlap between clusters.
	As illustrated in the toy Example~\ref{ex1} below, the sum of the two desired clusters $H_{:,1}{H_{:,1}}^T$ and $H_{:,2}{H_{:,2}}^T$ is not equal to the input matrix $A$.
	In the case where the diagonal entries are not taken into account, then the decomposition of $A$ into 
	$H_{:,1}{H_{:,1}}^T+H_{:,2}{H_{:,2}}^T$ is exact in the sense that $\|A-HH^T\|_{\text{OD},p} = 0$.
	Since a diagonal entry represents the similarity between an item and itself, it should be a large value for most similarity measures. 
	In order to approximate these large values, the optimization in symNMF methods deteriorates the quality of the cluster detection (see Section~\ref{Numeric} where we show that ignoring the diagonal entries leads to a better clustering accuracy).

	\begin{example}\label{ex1} For the matrix 
		\[
		\underbrace{
		\left(\begin{array}{*{3}{c}}
			1 & 1 & 0\\
			1 & 1 & 1\\
			0 & 1 & 1
			\end{array}\right)}_{A} 
			\approx
			\underbrace{
			\left(
			\begin{array}{*{3}{c}}
			1 & 1 & 0\\
			1 & 1 & 0\\
			0 & 0 & 0
			\end{array}
			\right)
			}_{H_{:,1}{H_{:,1}}^T}
		+
		\underbrace{
		\left(
			\begin{array}{*{3}{c}}
			0 & 0 & 0\\
			0 & 1 & 1\\
			0 & 1 & 1
			\end{array}
			\right)
			}_{ H_{:,2}{H_{:,2}}^T}, 
			\]
			symNMF is unable to perfectly recover these two clusters (in fact, one eigenvalue of $A$ is negative hence symNMF cannot exactly reconstruct this matrix even for $r$ larger than two; see below), 
			while  ODsymNMF perfectly does so as it does not take into account diagonal entries.  
	\end{example}

	\paragraph{From a theoretical point of view.} 
	The \textit{cp-rank} of a matrix $A$ is the minimum positive integer $r$ such that there exists an exact factorization $A=HH^T$ where $H$ is an $n$-by-$r$ nonnegative matrix~~\cite{abraham2003completely}.  
	The \textit{cp-rank} of a symmetric matrix $A$ is said to be infinite when no exact factorization $HH^T$ exists for any value of $r$. This is the case for the matrix $A$ in Example~\ref{ex1} since $A$ has one negative eigenvalue, namely, $1-\sqrt{2}$, while all approximations of the form $HH^T$ are positive definite hence 
	\[
	\min_{H \in \mathbb{R}^{n \times r}_+} \| A-HH^T\|_F \; \geq \;  \sqrt{2}-1 
	\] 
	for any value of  $r$ (this follows from the Eckart-Young theorem).  
	On the contrary, ODsymNMF is much more sound as there always exists an exact factorization with $H$ having at most $K$ columns where $K$ is half the number of non-zero off-diagonal entries of $A$. In particular, when $A$ has only positive off-diagonal entries, we have $K=\frac{n(n-1)}{2}$. 
	Such a factorization is obtained by using a column $H_{:,\ell}$ for each pair of entries $A_{pq} = A_{qp} \neq 0$ such that, for $i \neq j$, 
	\[
	(H_{:,\ell} {H_{:,\ell}}^T)_{ij}  
	= 
	\begin{cases}
	A_{pq} = A_{qp} & \text{if } (i,j) \in \{ (p,q), (q,p) \},\\
	0         & \text{otherwise, }
	\end{cases} 
	\]  
	which can be achieved for example by choosing  
	\[
	H_{i,\ell}= 
	\begin{cases}
	1& \text{if } i=p,\\
	A_{pq} & \text{if } i=q,\\
	0              & \text{otherwise.}
	\end{cases} 
	\]
This amounts to decompose $A$ as the sum of $K$ clusters containing $2$ elements corresponding to each pair of non-zero entries.

	\paragraph{From an algorithmic point of view.} 
	One of the most widely used optimization scheme in matrix factorization is CD which consists in updating one variable at a time while considering the other ones fixed~\cite{wright2015coordinate}. 
	When applied to symNMF, CD requires to find the minimum of a univariate quartic non-convex polynomial, which can be done in $O(1)$~\cite{vandaele2016efficient}. 
	However, as pointed out in \cite{ shi2017inexact,vandaele2016efficient}, the drawback is that the convergence to a stationary point is not guaranteed since the minimum of the \textit{quartic} polynomial may not be unique. 
	As we show in Section \ref{algorithms}, using ODsymNMF makes the optimization problem easier to solve: the sub-problem in one entry of $H$ (the others being fixed) is a \textit{quadratic} optimization problem over the nonnegative orthant for which a closed-form solution exists. Moreover, since the optimal solution of these sub-problems is uniquely attained, convergence of CD to stationary points is 
	guaranteed~\cite{Bertsekas99b,Bertsekas99}.  
	Moreover, as the sub-problems of ODsymNMF are simpler, we will be able to design CD for another loss function, namely the component-wise $\ell_1$-norm which would be highly non-trivial for symNMF (see Section~\ref{algo-l1}).

	\subsection{Rank-one binary ODsymNMF} \label{Proofrank1}
	
In many applications, the matrix $A$ is binary 	
	hence it is implicitly assumed that the noise is also binary~\cite{zhang2007binary}. 
	For a low-rank binary input matrix and binary noise, the maximum likelihood estimator is the optimal solution of 
	\begin{equation}\label{l0norm}
	\min_{H \in \mathbb{R}^{n \times r}_+} \quad ||A-HH^T||_{\text{OD},0}, 
	\end{equation} 
where the $\ell_0$ norm counts the number of non-zero entries in $A-HH^T$, that is, the numbre of mistmatches between $A$ and $HH^T$.  	
	An advantage of this formulation is that it produces binary solutions; see Lemma~\ref{lemmal0}.  
	Such binary solutions allow easier interpretations for most applications. 
	However, it is not straightforward to design local schemes for~\eqref{l0norm} since the objective function is of combinatorial nature.
	A standard approach to deal with \eqref{l0norm} is to replace it with its convex surrogate, the $\ell_1$-norm, where we also relax the binary constraints on $H$:
	\begin{equation}\label{l1norm}
	\min_{H \in [0,1]^{n \times r}_+}\quad \|A-HH^T\|_{\text{OD},1}.
	\end{equation}
	
In the following, we prove that the problems in $\ell_0$ and $\ell_1$ norms, that is, \eqref{l0norm} and \eqref{l1norm}, are equivalent for $r=1$; 
see Theorem~\ref{theoremequiv}. Note that this equivalence was also proved in the asymmetric case, that is, for NMF~\cite{gillis2018complexity}. 
This means that the $\ell_1$ norm is particularly well suited for binary input matrices, much better than the $\ell_2$ norm which generates dense solutions. In fact, in the rank-one case, the optimal solution using the $\ell_2$ norm is always positive when $A$ is irreducible (that is, when the graph induced by $A$ is connected) which follows from the Perron-Frobenius theorem~\cite{berman1994nonnegative}; see also~\cite{gillis2018complexity} for a discussion. 
	

	The first lemma shows that a solution of~\eqref{l0norm} can always be transformed into a binary solution with lower objective function value; this observation is similar than in the unsymmetric case~\cite[Lemma~1]{gillis2018complexity}. 
	\begin{lemma}
		Let $h\in \mathbb{R}^n$ and let $A$ be a $n$-by-$n$ binary matrix. Applying the following simple transformation to $h$  
		\[
		\Phi (h_i) = \begin{cases}
		0\quad\text{ if }h_i=0\\
		1\quad\text{ otherwise}
		\end{cases} , 
		\]
		gives
		\begin{equation*}
		||A-\Phi(h)\Phi(h)^T||_{\text{OD},0} \leq||A-hh^T||_{\text{OD},0}.
		\end{equation*}
		\label{lemmal0}
	\end{lemma}
	\begin{proof}
		There are two cases
		\begin{enumerate}
			\item If $h_ih_j = 0$, then $\Phi(h_i)\Phi(h_j) = 0$ hence the transformation does not affect the approximation.
			\item If $h_ih_j\neq 0$, then $\Phi(h_i)\Phi(h_j) = 1$.
			If $A_{ij} = 0$ then $||A-\Phi(h)\Phi(h)^T||_{\text{OD},0}=||A-hh^T||_{\text{OD},0}=1$ while, if $A_{ij}=1$, $||A-hh^T||_{\text{OD},0} \geq ||A-\Phi(h)\Phi(h)^T||_{\text{OD},0} = 0$.
		\end{enumerate}
	\end{proof}

	
Lemma~\ref{lemmal0} implies that the optimal solution of~\eqref{l0norm} with $r=1$ can be assumed to be binary without loss of generality, using a simple transformation.    
	The second lemma below shows that the same observation applies to~\eqref{l1norm}.   
	\begin{lemma} \label{lemmal1} 
Let $h\in [0,1]^n$ and let $A$ be a $n$-by-$n$ binary matrix. There exists a simple transformation to $h$ (see the proof below) that generates a binary vector $h' \in \{0,1\}^n$
		such that  
\[
||A-h'h'^T||_{\text{OD},1} \leq ||A-hh^T||_{\text{OD},1}. 
\]  
	\end{lemma}
	\begin{proof}
		Let $h \in [0,1]^n$, and let us show that we can transform it into a binary solution with lower objective function value. 
		For each $i\in\{1,...,n\}$ such that $h_i \notin \{0,1\}$, the terms of the objective function involving $h_i$ are 
		\begin{equation}
		f(h_i) = \sum_{j=1, j\neq i}^n \left|A_{ij} - h_i h_j\right|. 
		\label{eqnlemma2}
		\end{equation}
		The function \eqref{eqnlemma2} is piece-wise linear and convex hence minimizing it over the interval $[0,1]$ 
		leads to a global minimum equal to 0, 1, or one of the breakpoints $\frac{A_{ki}}{h_k}$ where 
		$k\in \{ j \ | \  j\neq i, h_j \neq 0\}$.
		Since $A$ is binary and $0 \leq h \leq 1$, we have that $\frac{A_{ki}}{h_k}$ is either equal to 0 or is larger than one.  Therefore, 0 or 1 is a global minimum of $f(h_i)$ over the interval $[0,1]$, and replacing $h_i$ by 0 or 1 will decrease the objective function. 
 	\end{proof}

Lemmas~\ref{lemmal0} and~\ref{lemmal1} imply that the $\ell_0$ and $\ell_1$ norm formulations of ODsymNMF are equivalent in the following sense. 
	
	\begin{theorem} \label{theoremequiv}
Any optimal solution 
of the rank-one problem~\eqref{l0norm} can be transformed into a binary optimal solution which is also optimal for the rank-one problem~\eqref{l1norm}, and vice versa. 
	\end{theorem}
	\begin{proof}
		By Lemmas~\ref{lemmal0} and~\ref{lemmal1}, we know that we can transform any solution into a binary solution with smaller objective function value. 
		For these binary solutions, the entries of the residual $P=A-hh^T$ belong to $\{-1,0,1\}$.  
		Since $\|P\|_{\text{OD},0}=\|P\|_{\text{OD},1}$ for any matrix $P\in\{-1,0,1\}^{n\times n}$, the binary optimal solutions for one problem are also optimal for the other problem. 
	\end{proof}
	
	Theorem~\ref{theoremequiv} shows that the $\ell_1$ relaxation~\eqref{l1norm} is particularly well suited for binary input matrices. In Section~\ref{algo-l1}, we design a CD scheme for this problem and, in Section~\ref{Numeric}, we illustrate this observation with some numerical experiments, showing that it outperforms the $\ell_2$ norm in this scenario.

	\section{Coordinate descent schemes for ODsymNMF}\label{algorithms}
	
	Coordinate descent (CD) is among the most intuitive methods to solve optimization problems~\cite{wright2015coordinate}.
	At each iteration, all variables are fixed but one which is then optimized exactly or inexactly depending on the difficulty of the corresponding univariate problem.
	For symNMF~\eqref{symmodel} using the Frobenius norm, when all entries of $H$ are fixed except one, the optimal value of the univariate problem is the root of a polynomial of the type $x^3+ax+b$ which can be computed in closed-form~\cite{vandaele2016efficient}. 
	
	Let us introduce our general CD framework for ODsymNMF. If we optimize the $(k,l)$th entry of $H$, 
	the univariate problem to solve is the following
	\begin{equation}
	\min_{H_{k,l}\geq 0} \left(\sum_{i\neq j}\left(A_{i,j}-\left(\sum_{t=1}^r H_{:,t}H_{:,t}^T\right)_{i,j}\right)^p\right)^{\frac{1}{p}}. \label{generalsymnmf}
	\end{equation}
	To simplify the presentation, let us focus on one rank-one factor, say $H_{:,l}{H_{:,l}}^T$, and denote $P$ the residual matrix $P=A-\sum_{t=1,t\neq l}^{r} H_{:,t}{H_{:,t}}^T$ corresponding to this factor. Let us also denote the vector $h = H(:,l)$. When optimizing the entries of $h = H(:,l)$ in CD, we face the following rank-one ODsymNMF problem: 
	\begin{equation}
	\min_{h\geq 0} \left(\sum_{i\neq j}\left(P_{i,j}-h_ih_j\right)^p\right)^{\frac{1}{p}}. \label{rank1symnmf}
	\end{equation}
	CD can be applied by solving iteratively rank-one ODsymNMF problems for each column $H_{:,l}$ with $l=1,...,r$ where the entries of each column are themselves solved via CD; see Algorithm \ref{algo1}. 
	
	\begin{algorithm}
		\caption{$\quad H = ODsymNMF(A,H_0)$}\label{algo1}
		\begin{algorithmic}[1]
			\item INPUT: $A\in \mathbb{R}^{n\times n}$, $H_0\in \mathbb{R}^{n\times r}_+$
			\item OUTPUT: $H\in \mathbb{R}^{n\times r}_+$
			\State $H \gets H_0$
			\State $R \gets A-HH^T$
			\While{stopping criterion not satisfied}
			\For{$l=1:r$}
			\State $P \gets R+H_{:,l}{H_{:,l}}^T$ \label{algo1-l7}
			\State $H_{:,l} \gets rank\_one\_ODsymNMF(P,H_{:,l})$ \label{algo1-l8}
			\State $R \gets P-H_{:,l}{H_{:,l}}^T$ \label{algo1-l9}
			\EndFor
			\EndWhile
		\end{algorithmic}
	\end{algorithm}
	\FloatBarrier
	
	It remains to show how to apply CD to rank-one ODsymNMF.  
	In the next two subsections, we will see how to do so  for the Frobenius-norm ($p=2$) and for the component-wise $\ell_1$-norm ($p=1$).

	\subsection{ODsymNMF with the Frobenius norm} \label{frobnorm}
	
	When $p=2$ in the optimization problem \eqref{rank1symnmf}, we are looking for the solution minimizing the least-squares error between $P$ and its rank-one approximation $hh^T$ without taking into account the diagonal entries.
	This problem can be written as
	\begin{equation}
	\min_{h\geq 0} f(h), \quad \text{ where } f(h) =  \frac{1}{4} \sum_{i=1}^n\sum_{\substack{j=1\\j\neq i}}^n\left(P_{i,j}-h_ih_j\right)^2. \label{rank1symnmffro}
	\end{equation}
	For the $k$th entry of $h$, with $k\in\{1,...,n\}$, the objective function can be decomposed as follows 
	\begin{equation}
	f(h)  = \frac{1}{4} \sum_{\substack{i=1\\i\neq k}}^n\sum_{\substack{j=1\\j\neq i\\j\neq k}}^n\left(P_{i,j}-h_ih_j\right)^2 + \frac{1}{4} \sum_{\substack{j=1\\j\neq k}}^n\left(P_{k,j}-h_kh_j\right)^2 + \frac{1}{4} \sum_{\substack{i=1\\i\neq k}}^n\left(P_{i,k}-h_ih_k\right)^2. \label{rank1symnmffrodecompose}
	\end{equation} 
	Since the matrix $P$ is symmetric, the last two terms of the right-hand side of \eqref{rank1symnmffrodecompose} are equal to one another. 
	This expression shows that the sub-problem in the entry $h_k$ is a quadratic optimization problem whose optimal solution is either $0$ or the single root of the equation ${\nabla f(h)}_k=0$ where $\nabla f(h)$ represents the gradient of $f(h)$.
	We have
	\begin{equation}  \label{gradient} 
	{\nabla f(h)}_k = \sum_{\substack{j=1\\j\neq k}}^n (h_kh_j^2-P_{k,j}h_j) = a_kh_k-b_k
	\end{equation} 
	where $a_k=\|h\|^2_2-h_k^2$ and $b_k=h^TP_{:,k} - h_kP_{k,k}$. 
	The optimal value $h_k^{+}$ that minimizes   (\ref{rank1symnmffrodecompose}) over the nonnegative orthant is
	\begin{align} \label{hk} 
	h_k^+ = \max\left(0,\frac{b_k}{a_k}\right). 
	\end{align} 
	Due to the computation of $a_k$ and $b_k$, the update of one variable with \eqref{hk} can be done in $\mathcal{O}(n)$.
	Therefore, Algorithm \ref{algo1} runs in $\mathcal{O}(n^2r)$ for updating once the $nr$ entries of $H$ since lines \ref{algo1-l7}, \ref{algo1-l8} and \ref{algo1-l9} run each in $\mathcal{O}(n^2)$.
	However, Algorithm \ref{algo1} requires to store the residual matrices $P$ and $R$ which have $\mathcal{O}(n^2)$ entries.
	Even when the matrix $A$ is sparse, these residual matrices  are usually dense which leads to a memory cost of $\mathcal{O}(n^2)$.  
	In the following, we show how to tackle the case of large sparse matrices more efficiently by avoiding the computation of  $P$ and $R$, reducing 
	the computational costs to $\mathcal{O}(Kr)$ and 
	the memory cost to 
	$\mathcal{O}(K)$ where $K$ is the number of nonzero entries of $A$.

\paragraph{Avoiding the explicit computation of the residual matrix} 
In order to compute~\eqref{hk}, we need to compute $a_k$ and $b_k$ that depend on $P$. 
After some calculations by simply expanding $P$, we obtain that the optimal solution for $h_{k,l}$, all other variables being fixed, is given by 
\[
h_{k,l}^+ = \max \left( 0 , \frac{b_{k,l}}{a_{k,l}} \right),  
\]  
where $a_{k,l} = \|H_{:,l}\|_2^2-H_{k,l}^2$ and 
\[ 
	b_{k,l} =  {H_{:,l}}^TA_{:,k}-{H_{:,l}}^T(HH^T)_{:,k}-H_{k,l}(A_{k,k}+H_{k,l}^2-\|H_{:,l}\|_2^2-\|H_{k,:}\|_2^2).
	\] 
	Algorithm~\ref{algo2} uses these expressions to avoid the computation of $P$ and $R$, but produces the same output as  Algorithm~\ref{algo1}.
	\begin{algorithm}
		\caption{$\quad H = \text{ODsymNMF-}\ell_2(A,H_0)$}\label{algo2}
		\begin{algorithmic}[1]
			\item INPUT: $A\in \mathbb{R}^{n\times n}$, $H_0\in \mathbb{R}^{n\times r}_+$
			\item OUTPUT: $H\in \mathbb{R}^{n\times r}_+$
			\State $H \gets H_0$
			\For{$l=1:r$} \label{algo3-line4}
			\State $C_l \gets ||H_{:,l}||_2^2$
			\EndFor
			\For{$k=1:n$}
			\State $L_k \gets ||H_{k,:}||_2^2$
			\EndFor \label{algo3-line9}
			\State $D \gets H^TH$ \label{algo3-line10}
			\While{stopping criterion not satisfied}
			\For{$l=1:r$}
			\For{$k=1:n$}
			\State $a_{k,l} \gets C_l-H_{k,l}^2$
			\State $b_{k,l} \gets (H_{:,l})^TA_{:,k} - H_{k,:}D_{:,l} + H_{k,l}(C_l + L_k - A_{k,k} - H_{k,l}^2)$
			\State $H_{k,l}^+ \gets \max(0,\frac{b_{k,l}}{a_{k,l}})$
			\State $C_l \gets C_l + (H_{k,l}^+)^2 - H_{k,l}^2$
			\State $L_k \gets L_k + (H_{k,l}^+)^2 - H_{k,l}^2$
			\State $D_{l,:}\gets D_{l,:} - H_{k,:}H_{k,l} + H_{k,:}^+H_{k,l}^+$
			\State $D_{:,l}\gets (D_{l,:})^T$
			\EndFor
			\EndFor
			\EndWhile
		\end{algorithmic}
	\end{algorithm}
	
	Let us analyse the computational cost and memory of  Algorithm~\ref{algo2}. 
	The precomputations of $\|H_{k,:}\|^2_2$, $\|H_{:,l}\|_2^2$ in $\mathcal{O}(nr)$ (see lines \ref{algo3-line4}-\ref{algo3-line9}) and of $D=$ in $\mathcal{O}(nr^2)$ (see line \ref{algo3-line10}) allow to compute the optimal value $H_{k,l}^+$ in $\mathcal{O}(n)$ when $A$ is dense due to the product ${H_{:,l}}^TA_{:,k}$.
	It is therefore possible to apply one iteration of CD in $\mathcal{O}(n^2r)$ operations.
	When $A$ contains $K$ nonzero entries, the computational complexity drops to $\mathcal{O}\left(r\max\left(K,nr\right)\right)$ since the computation of ${H_{:,l}}^TA$ can be done in
	$\mathcal{O}(K)$ operations.
	This result implies that that when $K=\mathcal{O}(n)$, which is the case for sparse matrices, 
	Algorithm~\ref{algo2} runs in $\mathcal{O}(nr^2)$ operations per iteration. 
	In terms of memory, Algorithm~\ref{algo2} only need to store $A$ and $H$, for a cost of $\mathcal{O}(K + nr)$.


	\subsection{ODsymNMF with the component-wise $\ell_1$-norm}\label{algo-l1}
	The $\ell_1$-norm is usually used to tackle Laplacian noise but is also a well-known surrogate of the $\ell_0$-norm in the presence of binary noise.
	In fact, we showed in Section \ref{Proofrank1} that for the ODsymNMF model using the $\ell_1$-norm is  equivalent to using the $\ell_0$-norm in the rank-one case. 
	For symNMF with the $\ell_1$-norm, the univariate problem arising when using CD is a sum of absolute value of quadratic terms. 
	Such a function is non-convex in general, making it difficult to optimize within a CD method. 
	With ODsymNMF, the quadratic terms disappear and we obtain a univariate convex problem. 
	When $p=1$ in~\eqref{rank1symnmf}, we have to minimize a sum of absolute values: 
	\begin{equation}
	\min_{h\geq 0} f(h), \quad \text{ where } f(h) =  \frac{1}{2} \sum_{i=1}^n\sum_{\substack{j=1\\j\neq i}}^n\left|P_{i,j}-h_ih_j\right|. \label{rank1symnmfl1}
	\end{equation}
	As with the $\ell_2$-norm, let us focus on the $k$th variable: we have 
	\begin{equation}
	f(h)  = \frac{1}{2} \sum_{\substack{i=1\\i\neq k}}^n\sum_{\substack{j=1\\j\neq i\\j\neq k}}^n\left|P_{i,j}-h_ih_j\right| + \frac{1}{2} \sum_{\substack{j=1\\j\neq k}}^n\left|P_{k,j}-h_kh_j\right| + \frac{1}{2} \sum_{\substack{i=1\\i\neq k}}^n\left|P_{i,k}-h_ih_k\right|. \label{rank1symnmfl1decompose}
	\end{equation} 
	The first term of the right-hand side of \eqref{rank1symnmfl1decompose} does not involve $h_k$, and the last two terms are equal when $P$ is symmetric.
	Hence the terms containing $h_k$ in the objective function $f(h)$ are $\sum_{{i=1,i\neq k}}^n\left|P_{i,k}-h_ih_k\right|$.
Hence, defining $a \in \mathbb{R}^{n - 1}$ as 
$a=h(\mathcal{K})$ where $\mathcal{K} = \{1,2,\dots,n\} \backslash \{k\}$, and $b \in \mathbb{R}^{n - 1}$ as 
$b = P(k,\mathcal{K})$, finding the optimal value of $h_k$ requires solving 
	\begin{equation} \label{weightedeqn} 
	\min_{x\geq 0}  \quad \sum_{i=1}^n |a_ix-b_i|. 
	\end{equation}
	The objective is a convex piecewise linear non-differentiable function, and this problem is a constrained weighted median problem. 
	There exists an algorithm in $\mathcal{O}(n)$ operations to solve the weighted median problem~\cite{gurwitz1990weighted}.  
	In the constrained case, because (\ref{weightedeqn}) is  convex, if the optimal solution $x^*$ is negative, 
	we can replace it by zero to obtain the optimal solution. 
	For the sake of completeness, 
	Algorithm~\ref{algo_cons_weightedmedian} in 
	Appendix~\ref{appendix_weighted_median} presents a simple algorithm for this constrained weighted median problem running in $\mathcal{O}(n \log n)$ operations (as it requires sorting the entries of a vector of length $n$). 
	Finally, 
	Algorithm~\ref{algo3} summarizes our algorithm for the rank-one ODsymNMF with $\ell_1$-norm.  
	
	\begin{algorithm}
		\caption{$\quad h = rank\_one\text{ODsymNMF-}\ell_1(P,h_0)$}\label{algo3}
		\begin{algorithmic}[1]
			\item INPUT: $P\in \mathbb{R}^{n\times n}$, $h_0\in \mathbb{R}^{n}_+$
			\item OUTPUT: $h\in \mathbb{R}^{n}_+$
			\State $h \gets h_0$
			\For{$k=1:n$}
			\State $\mathcal{K} = \{1,2,\dots,n\} \backslash \{k\}$ 
			\State $a \gets h(\mathcal{K})$  
			\State $b \gets P(k,\mathcal{K})$ 
			\State $h_k^+ \gets \texttt{constrained\_weighted\_median}(a,b)$ \% see Algorithm \ref{algo_cons_weightedmedian}
			\EndFor
		\end{algorithmic}
	\end{algorithm}
	\FloatBarrier
	
	Since Algorithm~\ref{algo_cons_weightedmedian} requires $\mathcal{O}(n \log n)$ operations, Algorithm \ref{algo3} runs in $\mathcal{O}(n^2 \log n)$. As noted above, the $\log n$ factor could be removed by using the weighted median algorithm from~\cite{gurwitz1990weighted}.  
	Overall, updating once each entry of $H$ using the $\ell_1$-norm when the residual matrix $R$ is available has a computational complexity of $\mathcal{O}\left(n^2r \log n\right)$ operations which is the same as for the $\ell_2$-norm, up to the logarithmic factor.

	\paragraph{Avoiding the explicit computation of the residual matrices}

As for the $\ell_2$ norm, in case of a sparse input matrix $A$, we would like to avoid the computation of the residual matrices $P$ and $R$. 
Similarly as for the $\ell_2$ norm, we substitute the expression $P=A-\sum_{t=1,t\neq l}^{r} H_{:,t}{H_{:,t}}^T$ in the updates of $H_{k,l}$; see Algorithm~\ref{algo4} for the details. 
Since the residual matrix is not stored, the main difference lies in the computation of the terms $\sum_{t=1,t\neq l}^{r} H_{i,t}H_{k,t}$ which occurs $\mathcal{O}(n)$ times for the update of one entry. 
	The overall computational complexity of Algorithm~\ref{algo4} is therefore $\mathcal{O}(n^2r^2)$ operations.
	As opposed to Algorithms \ref{algo1} and \ref{algo3} running in $\mathcal{O}(n^2r)$ operations, avoiding the storage of a $n$-by-$n$ dense matrix increases the computational cost.
	Moreover, unfortunately, the $\ell_1$-norm does not allow the sparsity of the input matrix $A$ to have any kind of effect in the overall complexity because $A$ is never multiplied by any other matrix during the updates (see lines \ref{algol3-line5}-\ref{algol3-line15} in Algorithm \ref{algo4}). 
	In summary, we can reduce the memory cost to $\mathcal{O}(K)$, while the computational cost slightly increases, to $\mathcal{O}(n^2r^2)$ operations. 
	
	\begin{algorithm}
		\caption{$\quad H = \text{ODsymNMF-}\ell_1(A,H_0)$}\label{algo4}
		\begin{algorithmic}[1]
			\item INPUT: $A\in \mathbb{R}^{n\times n}$, $H_0\in \mathbb{R}^{n\times r}_+$
			\item OUTPUT: $H\in \mathbb{R}^{n\times r}_+$
			\State $H \gets H_0$
			\While{stopping criterion not satisfied}
			\For{$l=1:r$} \label{algol3-line5}
			\For{$k=1:n$}
			\For{$i=1:n$}
			\State $a_i \gets H_{i,l}$
			\State $b_i \gets A_{i,k} - H_{k,:}{H_{i,:}}^T + H_{i,l}H_{k,l}$
			
			\EndFor
			\State $\mathcal{K} = \{1,2,\dots,n\} \backslash \{k\}$
			\State $H_{k,l}^+\gets \texttt{constrained\_weighted\_median}(a(\mathcal{K}),b(\mathcal{K}))$ \% see Algorithm \ref{algo_cons_weightedmedian}
			\EndFor
			\EndFor \label{algol3-line15}
			\EndWhile
		\end{algorithmic}
	\end{algorithm}
	\FloatBarrier
	
	
\subsection{Summary and convergence of the algorithms}
	
	Table~\ref{tab:my_label} summarizes the complexity of the algorithms proposed in this section. 
	The three lines of the table concerns respectively:
	\begin{itemize}
		\item the problem \eqref{odgeneral} for $p=2$, denoted ODsymNMF-$\ell_2$ and solved with Algorithm \ref{algo2},
		\item the problem \eqref{odgeneral} for $p=1$, denoted ODsymNMF-$\ell_1$ and solved with Algorithms \ref{algo1} and \ref{algo3} where a residual matrix is used,
		\item the problem \eqref{odgeneral} for $p=1$, denoted ODsymNMF-$\ell_1$ and solved with Algorithm \ref{algo4} where the use of a residual matrix is avoided.
	\end{itemize}

\begin{table}[ht]
	
	\centerline{
		\begin{tabular}{c||cc|cc|cc}
			\multicolumn{1}{c||}{\multirow{3}{*}{}} & \multicolumn{2}{c|}{General form} & \multicolumn{2}{c|}{Dense case} & \multicolumn{2}{c}{Sparse case}\\
			&&&  \multicolumn{2}{c|}{$K = $ \comp{n^2}}   &\multicolumn{2}{c}{$K = $ \comp{n}}\\			
			\cline{2-7}
			&\# flops & memory &\# flops & memory &\# flops & memory  \\
			\hline
			ODsymNMF-$\ell_2$  & \multirow{2}{*}{\comp{r\max(K,nr)}} & \multirow{2}{*}{\comp{\max(K,nr)}} & \multirow{2}{*}{\comp{n^2r}} & \multirow{2}{*}{\comp{n^2}} & \multirow{2}{*}{\comp{nr^2}} & \multirow{2}{*}{\comp{nr}} \\
			Algo. \ref{algo2} &&&&&&\\
			\hline
			ODsymNMF-$\ell_1$   & \multirow{2}{*}{\comp{n^2r}} & \multirow{2}{*}{\comp{n^2}} & \multirow{2}{*}{\comp{n^2r}} & \multirow{2}{*}{\comp{n^2}} & \multirow{2}{*}{\comp{n^2r}} & \multirow{2}{*}{\comp{n^2}}\\
			Algo. \ref{algo1} and \ref{algo3} &&&&&&\\
			\hline
			ODsymNMF-$\ell_1$  & \multirow{2}{*}{\comp{n^2r^2}} & \multirow{2}{*}{\comp{\max(K,nr)}} & \multirow{2}{*}{\comp{n^2r^2}} & \multirow{2}{*}{\comp{n^2}} & \multirow{2}{*}{\comp{n^2r^2}} & \multirow{2}{*}{\comp{nr}}\\
			Algo. \ref{algo4} &&&&&&
	\end{tabular}}
	\caption{Summary of the computational and memory  complexities.
	}
	\label{tab:my_label}
\end{table}


\paragraph{Convergence} The result~\cite[Proposition 2.7.1]{Bertsekas99, Bertsekas99b} guarantees that every limit point of an exact cyclic CD is a stationary point, given that 

\begin{enumerate}

\item the objective function is continuously differentiable,  

\item each block of variables is required to belong to a closed convex set,

\item the minimum computed at each iteration for a given block of variables is uniquely attained, and 

\item the objective function values in the interval between all iterates and the next (which is obtained by updating a single block of variables) is monotonically decreasing. 	
	
	\end{enumerate}
	
	For the $\ell_2$ norm, the subproblems in one variable are quadratic problems in one variable (see above) hence the four conditions above are satisfied. Therefore every limit point of Algorithm~\ref{algo2} is a stationary point. Note that there is at least one limit point since Algorithm~\ref{algo2} decreases the objective function monotonocally, 
	and the level sets of ODsymNMF-$\ell_2$, that is, $\{ H \geq 0 \ | \ \|A - HH^T \|_{\text{OD},2} \leq c\}$ for some constant $c$, are compact (Bolzano-Weierstrass theorem).

 For the $\ell_1$ norm, differentiability does not hold hence we can only guarantee the convergence of the objective function values (which decreases monotonically and is bounded below), as well as the existence of a limit point, as for the $\ell_2$ norm. 
In fact, for non-differentiable objectives, counter examples exist even when all the other assumptions above are satisfied; 
see for example~\cite[Example 14.5]{beck2017first}.

\section{A new initialization scheme for ODsymNMF}\label{Initialization}
	
	
	
	In this section we discuss the initialization of ODsymNMF algorithms, and propose a new very efficient greedy initialization scheme. 
	
	As far as we know, the only two strategies to initialize $H$ in symNMF are either at 
	random~\cite{kuang2015symnmf}  (e.g., using the uniform distribution in the interval [0,1] for each entry of $H$) or 
	with the zero matrix of appropriate dimension~\cite{vandaele2016efficient}. However, these initializations have some drawbacks: 
	\begin{itemize}
	
		\item When initializing $H$ randomly, the first iterates are trying at first to approximate a matrix which is highly perturbed ($A - \sum_{k=2}^{r} H_{:k}H_{:k}^T$) with a randomly generated matrix ($H_{:1}H_{:1}^T$) which is not very reasonable. Hence the first steps are wasting the global computational effort. 
		
		
		\item When initializing $H$ with the zero matrix, the solution found has a particular structure where the first factor is dense and the other ones are sparser.
		The reason is that the first factor is given more importance since it is optimized first hence it will be close to the best rank-one approximation of $A$~\cite{vandaele2016efficient}. 
		
	\end{itemize}


We propose in the following a greedy strategy that adapts to the norm used, that does not have the drawbacks mentioned above while having a low computational cost (roughly $r$ iterations of our CD methods). 
It consists in constructing each column of $H$ sequentially by selecting non-zero entries  depending on the non-zero entries of $A$. Our approach is summarized in Algorithm~\ref{greedyinit}. 
It works as follows, the matrix $H$ is inialized with the zero matrix, and a residual matrix $R$ is initialized as the input matrix $A$ and will be updated after each column of $H$ is constructed (steps 3 and 4). 
The columns of $H$ are computed sequentially by repeating the following steps: for $j=1,...,r$, 
	
	\begin{itemize}
	
		\item Initialization. The weighting vector $w$ is set as the vector of all ones of dimension $n$, and the index set $J$ as the empty set (steps 6 and 7). The vector $w$ will represent the importance of the entries of $H(:,j)$ while the index set $J$ will correspond to the non-zero entries of $H(:,j)$. 
		
		\item Loop over $\{1,2,\dots,n\}$: 

\begin{itemize}

\item Find the most important element, denoted $k$: it is the element maximizing $Rw$, that is, $k = \argmax_j (Rw)_j$ (step 11). 
Note that the very first time this loop is entered, $Rw$ is the sum of the entries in the rows of $A$ so the first element selected is the element corresponding to the row of $A$ with the largest $\ell_1$ norm. 
This element $k$ is added to $J$ (step 12).

\item The entry $H_{k,j}$ is then updated optimally by taking into account the information already contained in the cluster, that is, $R_{J,J}$ and $H_{J,j}$.
		That way the value obtained for $H_{k,j}$ is based on the values already updated, and is optimized according with the closed-form solutions derived in the previous sections (step 17).    
		For the first updated entry of a column, we set $H_{k,j}=1$ (step 14).
		
		\item The weighting vector $w$ is updated:  it is equal to the sum of the columns of $A$ in the index set $J$, that is, $w = \sum_j A(:,J)$ (step 15 or 18).  
		This step is particularly important in the first few iterations because the weighting vector $w$ has a strong impact on the selection process. It allows to add indices in $J$ highly connected to the indices already in $J$.

\end{itemize}		
		
		\item The residual is updated (step 21). 
		
	\end{itemize}

	
	
	\begin{algorithm}
		\caption{$\quad H = Greedy\_init(A,r)$}\label{greedyinit}
		\begin{algorithmic}[1]
			\item INPUT: $A\in \mathbb{R}^{n\times n}$, $r\in \mathbb{N}^+$
			\item OUTPUT: $H\in \mathbb{R}^{n\times r}_+$
			\State $H\gets 0^{n\times r}$
			\State $R\gets A$
			\For{$j=1:r$}
			\State $w\gets 1^n$
			\State $J \gets \{\}$
			\For{$i=1:n$}
			\State $s\gets Rw$ \label{RW}
			\State $s_J\gets -\infty$
			\State $[m,k]\gets max(s)$
			\State $J \gets J \cup \{k\}$ 
			\If{$i=1$}
			\State $H_{k,j}\gets 1$
			\State $w\gets A_{:,k}$
			\Else
			\State $H_{k,j}\gets Optimize\_hk(R_{J,J},H_{J,j},k)$
			\State $w\gets w + A_{:,k}$ \label{updateW}
			\EndIf
			\EndFor
			\State $R\gets R-H_{:,j}(H_{:,j})^T$
			\EndFor
		\end{algorithmic}
	\end{algorithm}

\noindent
The computational complexity of Algorithm \ref{greedyinit} is \comp{rn^3}, which makes it too expensive in most applications; and this is not desirable: the initialization scheme should have a low computational cost compared to the optimization scheme. 
This heavy computational cost comes from step \ref{RW} where the computation of the product between $R$ and $w$ is \comp{n^2}.
However, during the last iterations, it becomes less and less necessary to compute this product since the update of the weighting vector $w$ in step \ref{updateW} has a diluted effect, rendering the step \ref{RW} less useful. 
	Therefore, we found that updating $w$ only a small number of times works as well in practice. In particular, using a multiple of the rank $r$ (typically $2r$) works very well.  
Moreover, to keep the spatial complexity low, the residual matrix $R$ is not computed explicitly (as in our CD schemes).
These two changes lead to Algorithm~\ref{greedyinit2}: it is the same as Algorithm~\ref{greedyinit} except that $s$ is only updated for $2r$ iterations, while $R$ is not computed explicitly.  
	\begin{algorithm}
		\caption{$\quad H = Greedy\_init(A,r,p)$}\label{greedyinit2}
		\begin{algorithmic}[1]
			\item INPUT: $A\in \mathbb{R}^{n\times n}$, $r\in \mathbb{N}^+$, $p\in \{1,2\}$
			\item OUTPUT: $H\in \mathbb{R}^{n\times r}_+$
			\State $H\gets 0^{n\times r}$
			\For{$j=1:r$}
			\State $w\gets 1^n$
			\State $J \gets FALSE^n$
			\State $C_j \gets 0$
			\For{$i=1:n$}
			\If{$i<2r$}
			\State $s\gets Aw - H_{:,1:j-1} \left(H_{:,1:j-1}^Tw\right)$ \label{step10newalgo}
			\State $s_J\gets -\infty$
			\Else
			\State $s_J \gets -\infty$ \label{step13newalgo}
			\EndIf
			\State $[m,k]\gets max(s)$
			
			\If{$i=1$}
			\State $H_{k,j}\gets 1$
			\State $w\gets A_{:,k}$
			\Else
			\If{$p=2$}
			\State $b \gets H_{J,j}^TA_{J,k} - \left(H_{J,j}^TH_{J,1:j-1}\right)H_{k,1:j-1}^T $ \label{step21newalgo}
			\If{$b>0$}
			\State $H_{k,j}^+ \gets \frac{b}{C_j}$
			\Else
			\State $H_{k,j}^+\gets 0$
			\EndIf \label{step26newalgo}
			\Else
			\State $R_{J,k} \gets A_{J,k} - H_{J,1:j-1}H_{k,1:j-1}^T$ \label{step28newalgo}
			\State $H_{k,j}^+ \gets \texttt{constrained\_weighted\_median}(R_{J,k},H_{J,j})$ \%see Algorithm \ref{algo_cons_weightedmedian} \label{step29newalgo}
			
			\EndIf
			\State $w\gets w + A_{:,k}$
			\EndIf
			\State $J_k\gets TRUE$
			\State $C_j \gets C_j + H_{k,j}^2$
			\EndFor
			\EndFor
		\end{algorithmic}
	\end{algorithm}

\noindent
	The overall computational cost of Algorithm~\ref{greedyinit2} is \comp{r^2n^2} operations so it represents $r$ iterations of Algorithms \ref{algo2} and \ref{algo3}.
	Unfortunately the sparsity of the input has no impact on the computational cost, but the replacing of the residual matrix reduces the spatial complexity from \comp{n^2} to \comp{K}.

	\section{Numerical experiments}\label{Numeric}
	
	In this section we will compare the performances of the CD methods designed in Section \ref{algorithms} with the CD method for the usual symNMF model~\cite{vandaele2016efficient} 
	on synthetic and real examples. 
	Our code is available from \url{https://sites.google.com/site/nicolasgillis/code} and the numerical examples presented below can be directly run from this online code.
All tests are preformed using Matlab 
R2018a on a laptop Intel CORE i5-5200U CPU @2.2GHz 8Go RAM.

	\subsection{Synthetic examples} \label{synthexam}

	The main goal of the tests on synthetic examples is to show the robustness of the $\ell_1$-norm ODsymNMF when binary noise is added and the effectiveness of the greedy initialization proposed in Section \ref{Initialization}.
	The different experimental setups used are the following:
	
	\begin{itemize}
		\item \textit{Algorithms.} We compare our $\ell_2$-norm and $\ell_1$-norm ODsymNMF algorithms with the symNMF algorithm of \cite{vandaele2016efficient}.
		\item \textit{Initialization.} We compare the greedy initialization described in Section \ref{Initialization} with the zero and random initializations.
		\item \textit{Benchmark matrices.} The idea is to start from an input matrix for which the clustering solution $H^*$ is known and then add binary noise to that matrix.
		The benchmark matrices used are composed of multiple clusters of balanced sizes, that is, the matrix $A$ is a block diagonal matrix whose blocks have different size and are made up of all ones:  
		\[
	A = \left[\begin{array}{*{7}{c}}
	\textcolor{red}{1}&\textcolor{red}{\dots}&\textcolor{red}{1}&0&\dots&&\\
	\textcolor{red}{\vdots}&\textcolor{red}{\ddots}&\textcolor{red}{\vdots}&\vdots&&&\\
	\textcolor{red}{1}&\textcolor{red}{\dots}&\textcolor{red}{1}&0&\dots&&\\
	0&\dots&0&\textcolor{blue}{1}&\textcolor{blue}{\dots}&\textcolor{blue}{1}&\\
	\vdots&&\vdots&\textcolor{blue}{\vdots}&\textcolor{blue}{\ddots}&\textcolor{blue}{\vdots}&\\
	&&&\textcolor{blue}{1}&\textcolor{blue}{\dots}&\textcolor{blue}{1}&\\
	&&&&&&\ddots\\
	\end{array}\right] 
	\text{ with }
	H^* = \left[\begin{array}{*{3}{c}}
	\textcolor{red}{1}&0&\dots\\
	\textcolor{red}{\vdots}&\vdots&\\
	\textcolor{red}{1}&0&\dots\\
	0&\textcolor{blue}{1}&\\
	\vdots&\textcolor{blue}{\vdots}&\\
	&\textcolor{blue}{1}&\\
	&&\ddots
	\end{array}\right] .
	\]
	\Li
	To generate such matrices, we need the sizes of the clusters which we store in the vector $S$. For example, $S = [10\  10\ 5]$ means $A$ contains 2 cliques of size 10 each and a clique of size 5 so that $A$ is a $25$-by-$25$ binary matrix.

		\item \textit{Evaluation metric.} The availability of the ground-truth $H^*$ allows us to quantify the performance of a clustering algorithm. 
	We use a variation of the metric described in \cite{pompili2013onp} that quantifies the level of correspondence between the clusters found $H$ and the ground truth: 
	\begin{equation}
	Accuracy\ =\ 1- \max_{P\in [1,2,\dots,k]} \sqrt{\frac{||H_P - H^*||^2_F}{rn}} \in [0,1], 
	\label{accuracy}
	\end{equation} 
	where $[1,2,\dots,k]$ is the set of permutations of $\{1,2,\dots,k\}$ and $H_P$ is the matrix $H$ whose columns are rearranged according to the permutation $P$.

	\end{itemize}

	When the binary noise added is random, the experiment is repeated 10 times and we report the average accuracy of the solutions computed;  
	the is also done when the random initialization is used.

	\subsubsection{Random binary noise}

In this first experiment, we use 10 clusters of size 10, and the noise level $\delta \in [0,1]$ is  the probability to perturb an entry of $A$. In other words, for each entry of $A$, there is a probability of $\delta$ that this entry is flipped (from 1 to 0, and vice versa). 	

Table~\ref{table_noise} reports the accuracy when the noise level is fixed to $10\%$.  
For the random initialization, symNMF and ODsymNMF-$\ell_2$ perform similarly while ODsymNMF-$\ell_1$ performs badly. The reason is that ODsymNMF-$\ell_1$ is much more sensitive to initialization because it is intrinsically a more difficult problem (for $r=1$, it is NP-hard, which is not the case for symNMF). 
For the greedy initialization, ODsymNMF-$\ell_1$ outperforms symNMF and ODsymNMF-$\ell_2$. This was expected as ODsymNMF-$\ell_1$ is a better model in this scenario (Section~\ref{Proofrank1}), given that we can provide a good initial solution which is made possible through the greedy initialization. 
Moreover, the greedy initialization leads symNMF and ODsymNMF-$\ell_2$ to similar or better results. For this reason, we only keep the greedy initialization for the remainder of our numerical experiments.

\begin{table}[ht]	
	\centerline{
		\begin{tabular}{|cc| ccc| ccc|ccc|}
			\hline
			\multicolumn{2}{|c|}{\multirow{3}{*}{\textit{10$\%$ random noise}}}&\multicolumn{9}{c|}{Initialization type}\\
			\cline{3-11}
			&&\multicolumn{3}{c|}{Random}&\multicolumn{3}{c|}{Zero}&\multicolumn{3}{c|}{Greedy}\\
			\cline{3-11}
			&&OD-$\ell_1$&OD-$\ell_2$&Sym&OD-$\ell_1$&OD-$\ell_2$&Sym&OD-$\ell_1$&OD-$\ell_2$&Sym\\
			\hline
			\multicolumn{1}{|c|}{\multirow{2}{*}{2 clusters}}&balanced&55& 91& \textbf{91}&51& 29& \textbf{91}&\textbf{98}& 91& 91\\
			\cline{2-11}
			\multicolumn{1}{|c|}{}&unbalanced&60& 87& \textbf{88}&62& 29& \textbf{85}&\textbf{94}& 87& 88\\
			\hline
			\multicolumn{1}{|c|}{\multirow{2}{*}{5 clusters}}&balanced&64& 90& \textbf{90}&64& 88& \textbf{90}&\textbf{96}& 90& 90\\
			\cline{2-11}\multicolumn{1}{|c|}{}&unbalanced&66& 89 &\textbf{89}&64& 55& \textbf{90}&\textbf{96}& 90& 90\\
			\hline
			\multicolumn{1}{|c|}{\multirow{2}{*}{10 clusters}}&balanced&76& 90& \textbf{90}&73& 68 &\textbf{90}&\textbf{98}& 90 &90\\
			\cline{2-11}\multicolumn{1}{|c|}{}&unbalanced&78& 86& \textbf{88}&73& 68& \textbf{88}&\textbf{93}& 89& 88\\
			\hline
	\end{tabular} }
	\caption{Summary of the accuracies obtained for the three types of initialization in each problem. OD-$\ell_1$ stands for ODsymNMF-$\ell_1$, OD-$\ell_2$ stands for ODsymNMF-$\ell_2$ and Sym stands for symNMF. The highest accuracy for each initialization type is bolded.}	
	\label{table_noise}
\end{table}

	Figure \ref{randomnoise} provides the accuracy for the different models depending on the noise level. 
	For low levels of noise ($\delta \leq 0.15$), ODsymNMF-$\ell_1$ recovers a very good clustering (accuracy above 90\%). For larger noise levels, the performances deteriorate rapidly. 
	As expected ODsymNMF-$\ell_2$ and symNMF perform similarly since the difference between these two models is the diagonal of $A$ composed of $n$ elements. However, recall that  ODsymNMF-$\ell_2$ is computationally cheaper and has convergence guarantee. 
	
	
	\begin{figure}[!h]
		\centering
		\includegraphics[width=0.6\textwidth]{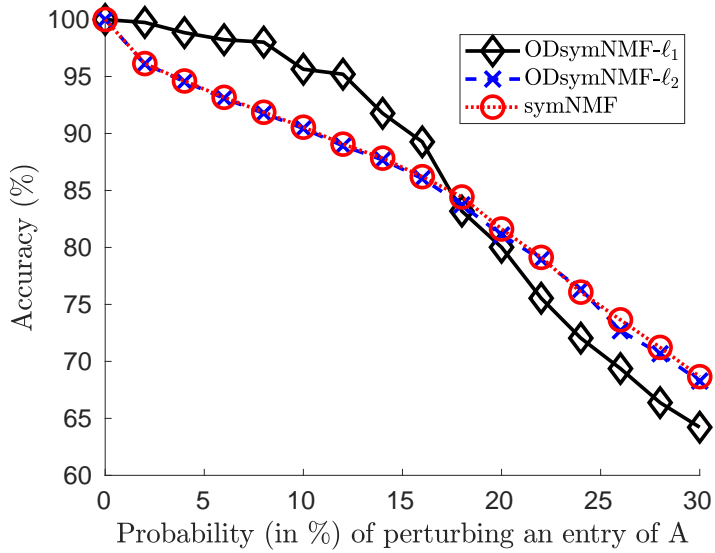}
		\caption{Evolution of the accuracy when random binary noise is added. The input matrix is composed of 10 cliques of size 10 each.}
		\label{randomnoise}
	\end{figure}
	\FloatBarrier

	\subsubsection{Adversarial binary noise}

In order to highlight the differences between ODsymNMF-$\ell_1$ and ODsymNMF-$\ell_2$, let us construct the following adversarial example: 
	\[
	A = \left[\begin{array}{*{9}{c}}
	\textcolor{red}{1}&\textcolor{red}{\dots}&\textcolor{red}{1}&0&\dots&0&{0}&{\dots}&{0}\\
	
	\textcolor{red}{\vdots}&\textcolor{red}{\ddots}&\textcolor{red}{\vdots}&\vdots&\ddots&\vdots&{\vdots}&&{\vdots}\\
	
	\textcolor{red}{1}&\textcolor{red}{\dots}&\textcolor{red}{1}&0&\dots&0&{\vdots}&&{\vdots}\\
	
	0&\dots&0&\textcolor{blue}{1}&\textcolor{blue}{\dots}&\textcolor{blue}{1}&{\vdots}&&{\vdots}\\
	
	\vdots&\ddots&\vdots&\textcolor{blue}{\vdots}&\textcolor{blue}{\ddots}&\textcolor{blue}{\vdots}&{\vdots}&&{\vdots}\\
	
	0&\dots&0&\textcolor{blue}{1}&\textcolor{blue}{\dots}&\textcolor{blue}{1}&{0}& {\dots}&{0}\\
	
	{0}&{\dots}&{\dots}&{\dots}&{\dots}&{0}&&&\\
	
	{\vdots}&&&&&{\vdots}&&\mathbb{I}_m&\\
	
	{0}&{\dots}&{\dots}&{\dots}&{\dots}&{0}&&&\\
	
	
	\end{array}\right], 
	\] 
where the  matrix $A$ is composed of two cliques and of $m$ isolated elements (identity matrix $\mathbb{I}_m$).
	The adversarial noise consists in adding connections between the cliques and the isolated elements. 
	Here the noise level is the number of connections added between an isolated element and the 2 cliques. 
	As long as this number of connections does not exceed half the size of the cliques, we can expect ODsymNMF-$\ell_1$ to recover the ground truth.
	This is in fact what is observed in Figure \ref{adversarialnoise}.
	
	\begin{figure}[!h]
		\centering
		\includegraphics[width=0.6\textwidth]{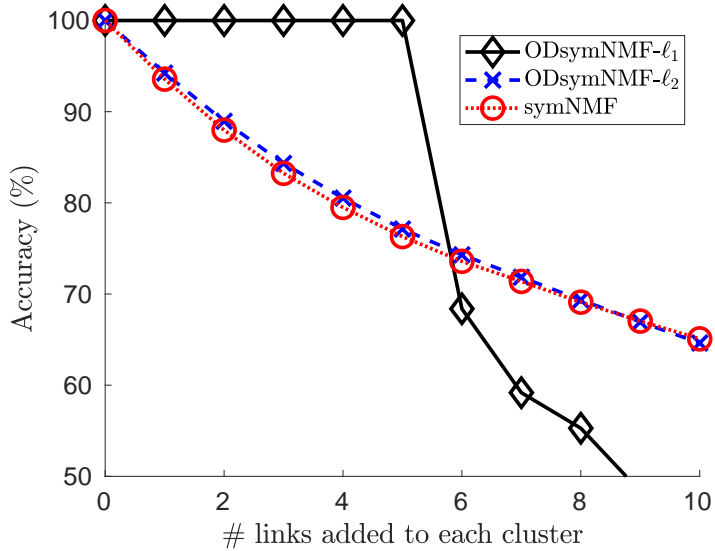}
		\caption{Evolution of the accuracy when adversarial binary noise is added. The input matrix is composed of 2 cliques of size 10 each and an identity matrix of size 10.}
		\label{adversarialnoise}
	\end{figure}
	\FloatBarrier

	\paragraph{Conclusions} 
	The conclusions from running these synthetic experiments is threefold:  the greedy initialization outperforms the zero and random initializations,  ODsymNMF-$\ell_1$ outperforms ODsymNMF-$\ell_2$ and symNMF in the presence of binary noise while ODsymNMF-$\ell_2$ and symNMF perform similarly.

	\subsection{Document data sets}

	We now perform clustering of real documents data sets. 
	These documents are represented as a word-count matrix $X\in \mathbb{N}^{n\times m}$; see Table \ref{tab:documents}. 
	\begin{table}[ht]
		\centering
		\small
		\begin{tabular}{|c|c|c|c|} 
			\hline
			Data set & $\#$ Documents ($=n$)&$\#$ Words ($=m$)& rank $r$\\
			\hline\hline
			classic&7094&41681&4\\
			ohscal&11162&11465&10\\
			hitech&2301&10080&6\\
			reviews&4069&18483&5\\
			sports&8580&14870&7\\
			la1&3204&31472&6\\
			la2&3075&31472&6\\
			k1b&2340&21839&6\\
			tr11&414&6429&9\\
			tr23&204&5832&6\\
			tr41&878&7454&10\\
			tr45&690&8261&10\\
			\hline
		\end{tabular}
		\caption{Data of 12 documents sets from~\cite{zhong2005generative}.}
		\label{tab:documents}
	\end{table}
	\FloatBarrier

	\paragraph{Similarity matrix A}
	
	In order to obtain a similarity matrix starting from the word-count matrix $X$,  we choose to use a simple but powerful one: the cosine angle.
	The similarity between documents $a$ and $b$ is then equal to $\frac{x_a^Tx_b}{\|x_a\|_2\|x_b\|_2}$.
	The values inside $A$ are therefore reals (not binary); this implies a lot of overlap hence a difficult problem. The `correct' clustering is known as the documents are sorted in categories.  
	We refer the reader to~\cite{strehl2000impact} for a discussion on the construction of the similarity matrix.

	\paragraph{Interpretation of the output $H$}
	
	Table~\ref{tab:results} reports the accuracy~\eqref{accuracy} of symNMF, ODsymNMF-$\ell_2$ and ODsymNMF-$\ell_1$.
	We used the greedy initialization for all algorithms.

	\begin{table}[ht]
		\centering
		\begin{tabular}{|c| c| c|c|}
			\hline
			Data & SymNMF&ODsymNMF-$\ell_2$ & ODsymNMF-$\ell_1$ \\
			\hline\hline
			classic&63.67&63.67&\textbf{66.33}\\
			ohscal&\textbf{43.24}&43.16&38.08\\
			hitech&49.07&49.24&\textbf{52.19}\\
			reviews&49.37&49.55&\textbf{70.07}\\
			sports&\textbf{51.46}&51.41&48.81\\			
			la1&\textbf{49.16}&48.81&40.61\\
			la2&\textbf{48.94}&48.62&39.45\\
			k1b&57.18&58.68&\textbf{66.45}\\
			tr11&59.66&\textbf{59.90}&51.21\\
			tr23&35.29&35.29&\textbf{36.76}\\
			tr41&46.70&\textbf{47.15}&47.04\\
			tr45&42.90&42.61&\textbf{43.04}\\
			\hline
		\end{tabular} 
		\caption{Accuracy (in $\%$) for each data set. The bold values are the best of each line.}
		\label{tab:results}
	\end{table}

We observe the following: 
\begin{itemize}

\item As for the synthetic data sets, SymNMF and ODsymNMF-$\ell_2$ perform similarly (but bare in mind that ODsymNMF-$\ell_2$ has numerical and theoretical advantages). 

\item ODsymNMF-$\ell_1$ performs very differently than SymNMF and ODsymNMF-$\ell_2$. In some cases, it provides a much better accuracy (in particular, for the reviews data set; from 50\% to 70\% accuracy) and, in other cases, a worse accuracy (from about 50\% to 40\% for la1 and la2). 
The reson why ODsymNMF-$\ell_1$ does not outperform SymNMF and ODsymNMF-$\ell_2$ is that the data sets are not binary, and do not follow the low-rank model very closely. However, 
it is interesting to observe that these models obtain rather different solutions hence should be used in different scenario depending on the (noise) model. 
Moreover, they are able to extract different clusters within data sets. Hence an interesting direction of research would be to combine $\ell_1$ and $\ell_2$ norm models (e.g., using an objective function which is a combinations of these two objectives as in~\cite{gillis2019distributionally}) 
to outperform ODsymNMF-$\ell_2$ and ODsymNMF-$\ell_1$ in all cases.

\end{itemize}	
	

	\section{Conclusion} 
	
	In this paper, we proposed a new meaningful model for symmetric nonnegative matrix factorization (symNMF) by discarding the diagonal elements; we refer to this model as off-diagonal symNMF (ODsymNMF). 
	This allowed us to design efficient coordinate descent algorithms for the $\ell_2$ norm and $\ell_1$ norm. 
For the $\ell_2$ norm, our algorithm has the advantage to be computationally cheaper than the CD method of symNMF~\cite{vandaele2016efficient} (the subproblems in one variable are quadratic instead of quartic) 
while having convergence guarantees. 
For the $\ell_1$ norm, this is, to the best of our knowledge, the first algorithm of this kind for symNMF. It was made possible precisely because we discarded the diagonal elements. This 
$\ell_1$-norm model is better suited for binary input matrices which we theoretically proved in Section~\ref{Proofrank1} in the rank-one case, and empirically illustrated in Section~\ref{synthexam} on synthetic data sets. 
We also provided numerical experiments for real document data sets, where the    $\ell_2$-norm and $\ell_1$-norm models perform rather differently. Future work includes the design of other initialization strategies, as well as 
 new symNMF-like models that would adapt to the structure of the input matrix for example using distributionally robust models as in~\cite{gillis2019distributionally}.

	\appendix
	
	\section{The constrained weighted median problem}
	\label{appendix_weighted_median}

Algorithm \ref{algo_cons_weightedmedian} provides a pseudocode to compute the solution to the constrained  weighted median problem 
\[
\min_{x \geq 0} \sum_i |a_i x - b_i|. 
\]
The algorithm works as follow:
\begin{itemize}

	\item the set $S$ of breakpoints $\frac{b_i}{a_i}$ is initialized for all $i=1,...,n$ such that $a_i\neq 0$ (because when $a_i=0$, the contribution of the $i$th term in the objective function is a constant)  and the vector $a$ is then sorted and normalized according to the values in $S$,
	
	\item as the values $a_i$ correspond to the slopes, the second step of the algorithm looks for the $k$th breakpoint for which we have 
	$\sum_{i=1}^{k-1} a_i < \sum_{i=k}^n a_i$ and 
	$\sum_{i=1}^k a_i \geq \sum_{i=k+1}^n a_i$. 
	It corresponds to a global optimum since the slope on the left is negative, and on the right is nonnegative. 
	
\end{itemize}
	\begin{algorithm}
		\caption{$\quad x = constrained\_weighted\_median(a,b)$}\label{algo_cons_weightedmedian}
		\begin{algorithmic}[1]
			\item INPUT: $a\in \mathbb{R}^{n}_+$, $b\in \mathbb{R}^{n}$
			\item OUTPUT: $x\in \mathbb{R}$
			\State $S\gets \emptyset$
			\For{$i=1:n$}
			\If{$a_i \neq 0$}
			\State $S\gets S\cup\{\frac{b_i}{a_i}\}$ \label{Scup}
			\EndIf
			\EndFor
			\State $[S,Inds]\gets sort(S)$
			\State $a \gets \frac{a(Inds)}{sum(a)}$
			\State $i \gets 1$
			\State $CumulatedSum \gets 0$
			\While{$CumulatedSum < 0.5$}
			\State $CumulatedSum \gets CumulatedSum + a_i$
			\State $x\gets S_i$
			\State $i\gets i+1$
			\EndWhile
			
		\end{algorithmic}
	\end{algorithm}
	\FloatBarrier
	
	\newpage 
	
	\bibliographystyle{spmpsci}
	\bibliography{ArticleODsymNMF}

\begin{thebibliography}{10}
\providecommand{\url}[1]{{#1}}
\providecommand{\urlprefix}{URL }
\expandafter\ifx\csname urlstyle\endcsname\relax
  \providecommand{\doi}[1]{DOI~\discretionary{}{}{}#1}\else
  \providecommand{\doi}{DOI~\discretionary{}{}{}\begingroup
  \urlstyle{rm}\Url}\fi

\bibitem{abraham2003completely}
Abraham, B., Naomi, S.M.: Completely positive matrices.
\newblock World Scientific (2003)

\bibitem{beck2017first}
Beck, A.: First-order methods in optimization, vol.~25.
\newblock SIAM (2017)

\bibitem{berman1994nonnegative}
Berman, A., Plemmons, R.J.: Nonnegative matrices in the mathematical sciences,
  vol.~9.
\newblock SIAM (1994)

\bibitem{Bertsekas99b}
Bertsekas, D.: Corrections for the book nonlinear programming (1999)

\bibitem{Bertsekas99}
Bertsekas, D.: Nonlinear Programming: Second Edition.
\newblock Athena Scientific, Massachusetts (1999)

\bibitem{borsdorf2010computing}
Borsdorf, R., Higham, N.J., Raydan, M.: Computing a nearest correlation matrix
  with factor structure.
\newblock SIAM Journal on Matrix Analysis and Applications \textbf{31}(5),
  2603--2622 (2010)

\bibitem{chen2008non}
Chen, Y., Rege, M., Dong, M., Hua, J.: Non-negative matrix factorization for
  semi-supervised data clustering.
\newblock Knowledge and Information Systems \textbf{17}(3), 355--379 (2008)

\bibitem{cichocki2009nonnegative}
Cichocki, A., Zdunek, R., Phan, A.H., Amari, S.i.: Nonnegative matrix and
  tensor factorizations: applications to exploratory multi-way data analysis
  and blind source separation.
\newblock John Wiley \& Sons (2009)

\bibitem{dickinson2014computational}
Dickinson, P.J., Gijben, L.: On the computational complexity of membership
  problems for the completely positive cone and its dual.
\newblock Computational optimization and applications \textbf{57}(2), 403--415
  (2014)

\bibitem{xiao2019uniq}
Fu, X., Huang, K., Sidiropoulos, N.D., Ma, W.K.: Nonnegative matrix
  factorization for signal and data analytics: Identifiability, algorithms, and
  applications.
\newblock IEEE Signal Processing Magazine \textbf{36}(2), 59--80 (2019)

\bibitem{gillis2014and}
Gillis, N.: The why and how of nonnegative matrix factorization.
\newblock In: J.~Suykens, M.~Signoretto, A.~Argyriou (eds.) Regularization,
  Optimization, Kernels, and Support Vector Machines, chap.~12, pp. 257--291.
  Chapman \& Hall/CRC, Boca Raton, Florida (2014)

\bibitem{gillis2019distributionally}
Gillis, N., Hien, L.T.K., Leplat, V., Tan, V.Y.: Distributionally robust and
  multi-objective nonnegative matrix factorization.
\newblock arXiv preprint arXiv:1901.10757  (2019)

\bibitem{gillis2018complexity}
Gillis, N., Vavasis, S.A.: On the complexity of robust pca and $\ell_1$-norm
  low-rank matrix approximation.
\newblock Mathematics of Operations Research \textbf{43}(4), 1072--1084 (2018)

\bibitem{gurwitz1990weighted}
Gurwitz, C.: Weighted median algorithms for $l_1$ approximation.
\newblock BIT \textbf{30}(2), 301--310 (1990)

\bibitem{huang2013non}
Huang, K., Sidiropoulos, N.D., Swami, A.: Non-negative matrix factorization
  revisited: Uniqueness and algorithm for symmetric decomposition.
\newblock IEEE Transactions on Signal Processing \textbf{62}(1), 211--224
  (2013)

\bibitem{kuang2012symmetric}
Kuang, D., Ding, C., Park, H.: Symmetric nonnegative matrix factorization for
  graph clustering.
\newblock In: Proceedings of the 2012 SIAM international conference on data
  mining, pp. 106--117. SIAM (2012)

\bibitem{kuang2015symnmf}
Kuang, D., Yun, S., Park, H.: Symnmf: nonnegative low-rank approximation of a
  similarity matrix for graph clustering.
\newblock Journal of Global Optimization \textbf{62}(3), 545--574 (2015)

\bibitem{long2007relational}
Long, B., Zhang, Z.M., Wu, X., Yu, P.S.: Relational clustering by symmetric
  convex coding.
\newblock In: Proceedings of the 24th international conference on Machine
  learning, pp. 569--576. ACM (2007)

\bibitem{pompili2013onp}
Pompili, F., Gillis, N., Glineur, F., Absil, P.A.: Onp-mf: An orthogonal
  nonnegative matrix factorization algorithm with application to clustering.
\newblock In: ESANN. Citeseer (2013)

\bibitem{shi2017inexact}
Shi, Q., Sun, H., Lu, S., Hong, M., Razaviyayn, M.: Inexact block coordinate
  descent methods for symmetric nonnegative matrix factorization.
\newblock IEEE Transactions on Signal Processing \textbf{65}(22), 5995--6008
  (2017)

\bibitem{strehl2000impact}
Strehl, A., Ghosh, J., Mooney, R.: Impact of similarity measures on web-page
  clustering.
\newblock In: Workshop on artificial intelligence for web search (AAAI 2000),
  vol.~58, p.~64 (2000)

\bibitem{vandaele2016efficient}
Vandaele, A., Gillis, N., Lei, Q., Zhong, K., Dhillon, I.: Efficient and
  non-convex coordinate descent for symmetric nonnegative matrix factorization.
\newblock IEEE Transactions on Signal Processing \textbf{64}(21), 5571--5584
  (2016)

\bibitem{wright2015coordinate}
Wright, S.J.: Coordinate descent algorithms.
\newblock Mathematical Programming \textbf{151}(1), 3--34 (2015)

\bibitem{yan2013learning}
Yan, X., Guo, J., Liu, S., Cheng, X., Wang, Y.: Learning topics in short texts
  by non-negative matrix factorization on term correlation matrix.
\newblock In: proceedings of the 2013 SIAM International Conference on Data
  Mining, pp. 749--757. SIAM (2013)

\bibitem{yang2012clustering}
Yang, Z., Hao, T., Dikmen, O., Chen, X., Oja, E.: Clustering by nonnegative
  matrix factorization using graph random walk.
\newblock In: Advances in Neural Information Processing Systems, pp. 1079--1087
  (2012)

\bibitem{zass2005unifying}
Zass, R., Shashua, A.: A unifying approach to hard and probabilistic
  clustering.
\newblock In: Tenth IEEE International Conference on Computer Vision (ICCV'05)
  Volume 1, vol.~1, pp. 294--301. IEEE (2005)

\bibitem{zhang2007binary}
Zhang, Z., Li, T., Ding, C., Zhang, X.: Binary matrix factorization with
  applications.
\newblock In: Seventh IEEE International Conference on Data Mining (ICDM 2007),
  pp. 391--400. IEEE (2007)

\bibitem{zhong2005generative}
Zhong, S., Ghosh, J.: Generative model-based document clustering: a comparative
  study.
\newblock Knowledge and Information Systems \textbf{8}(3), 374--384 (2005)

\end{thebibliography}

\end{document}